\documentclass[a4paper,12pt]{article}
\usepackage{amsfonts}
\usepackage[T1]{fontenc}
\usepackage{microtype}
\usepackage{makeidx}
\usepackage{setspace}
\usepackage{graphicx}
\usepackage{epstopdf}
\usepackage[all]{xy}
\usepackage{amsmath}
\usepackage{amssymb}
\usepackage{booktabs}
\usepackage{braket}
\usepackage{array}
\usepackage{tabularx}
\usepackage{multirow}
\usepackage{indentfirst}
\usepackage{cite}
\usepackage{stmaryrd}
\usepackage{faktor}
\usepackage{titling}
\usepackage{tikz}
\usepackage{dsfont}
\usetikzlibrary{matrix,arrows,decorations.pathmorphing}

\usepackage{amsthm}
\theoremstyle{plain}
\newtheorem{thm}{Theorem}[section]
\newtheorem{dfn}[thm]{Definition}
\newtheorem{lem}[thm]{Lemma}
\newtheorem{prop}[thm]{Proposition}
\newtheorem{cor}[thm]{Corollary}
\newtheorem{ex}[thm]{Example}

\theoremstyle{remark}
\newtheorem{oss}[thm]{Remark}

\DeclareMathOperator{\maj}{maj}
\DeclareMathOperator{\imaj}{maj^*}
\DeclareMathOperator{\nmaj}{nmaj}
\DeclareMathOperator{\inmaj}{nmaj^*}
\DeclareMathOperator{\fmaj}{fmaj}
\DeclareMathOperator{\dmaj}{dmaj}
\DeclareMathOperator{\Dmaj}{Dmaj}
\DeclareMathOperator{\imag}{Im}
\DeclareMathOperator{\inv}{inv}

\DeclareMathOperator{\negg}{Neg}

\begin{document}

\begin{center}

{\Large \bf Parabolically induced functions and equidistributed pairs}
 \vspace{0.8cm}

 Paolo Sentinelli \footnote{Supported by Postdoctorado FONDECYT-CONICYT 3160010.}
\\

Departamento de Matem\'aticas \\

Universidad de Chile\\

Las Palmeras 3425, \~Nu\~noa\\

00133 Santiago, Chile \\

{\em paolosentinelli@gmail.com } \\

\end{center}

\vspace{1cm}

\begin{abstract}
Given a function defined over a parabolic subgroup of a Coxeter group, equidistributed with the length, we give a procedure to construct a function over the entire group, equidistributed with the length. Such a procedure permits to define functions equidistributed with the length in all the finite Coxeter groups.  We can establish our results in the general setting of graded posets which satisfy some properties. These results apply to some known functions arising in Coxeter groups as the \emph{major index}, the \emph{negative major index} and the \emph{D-negative major index} defined in type $A$, $B$ and $D$ respectively.

\end{abstract}

\section{Introduction}

Let us consider the functions $(\ell+\maj) : S_n \rightarrow \mathbb{N}$ and $(\ell-\maj) : S_n \rightarrow \mathbb{Z}$, being
$S_n$ the group of permutations on $\{1,2,...,n\}$, $\ell$ the length function (or inversion number) and $\maj$ the major index. What can we say about the behavior of these functions? The question can be stated in a more general setting, when $\ell$ is replaced by the rank function of a finite graded poset and $\maj$ by any function on the poset equidistributed with its rank; this is done in Section \ref{sec2}, where also the concept of induced equidistributed function is introduced. There are some relations between these sums and differences and the fact that a function is induced by another; maybe the more notable one is that the image of the function defined as the difference of two equidistributed functions gives a necessary condition for a function to be induced by another in the sense of Definition \ref{funzione indotta} (see Theorem \ref{teorema-} and its implications in type $B$ and $D$, i.e. Propositions \ref{propB} and \ref{propD}). For example, by the computations \eqref{successione-} and \eqref{successione-fmaj} we can assert that the \emph{flag-major index} on $S_5^B$ is not induced by the major index on $S_k$ for all $k\leqslant 5$, neither by the flag-major index on $S_k^B$ for all $k<5$.

We can prove that some properties of the inducing function are inherited by the induced function. These are the existence of an involution relating the two distributions and the symmetry of the pair of distributions (see Theorems \ref{invol} and  \ref{simmetria}), which in fact are equivalent properties (Theorem \ref{thm symm}). Moreover we can state a weaker condition which two equidistributed functions related by an involution have to satisfy (Proposition \ref{criterioinv}). Such results permit to deduce directly some known symmetries in type $B$ and $D$ and to state the existence and the non-existence of an involution for the \emph{negative major index} and the \emph{D-negative major index} (existence, giving the involution explicitly) and for the \emph{flag-major index} and the $\emph{D-major index}$ (non-existence).

That the function $\ell$ and $\maj$ are equidistributed
is an old result due to MacMahon; the general concept of \emph{Mahonian pair} appeared in \cite{SaganSavage} applies to any pair of subsets
of the sets of finite sequences of positive integers. For results concerning multivariate statistics in the symmetric group involving the inversion number, the major index and other notable functions, see \cite{GarsiaGessel} and the references related. On the side of Coxeter groups, in analogy to the major index, some functions defined on classical Weyl groups and equidistributed with the length have been defined in the last decades (see e.g. \cite{AdinBrenti}, \cite{AdinRoichman},  \cite{Biagioli} and \cite{BiagiolCaselli}); for their relevance in representation theory see e.g. \cite{AdinBrenti2}, \cite{BiagiolCaselli}, \cite{GarsiaStanton} or the book \cite{bergeron}. We want to cite also the article \cite{carnevale} for other results concerning the negative statistics. That such negative statistics are induced by the major index of the symmetric group is shown in Sections \ref{sec4} and \ref{sec5}. A multivariate equidistribution in type $B$ is also proved, relating the flag-major index and the inverse negative major index (see Theorem \ref{ellinmaj} and Corollary \ref{cornmajfmaj}).
In Section \ref{other type} we discuss what can be obtained for Coxeter systems of other types.

\section{Notations and preliminaries} \label{sec1}

In this section we give some notation and we collect some basic
results in the theory of Coxeter groups
which will be useful in the sequel. The reader can consult
\cite{bjornerbrenti} and \cite{Hum} for further details. For general terminology about posets we follows \cite{StaEC1}. We want only to specify that the poset morphisms considered are the order preserving functions.

Denote $\mathbb{Z}$ the ring of integers, $\mathbb{N}$ the set of non-negative
integers, $\mathbb{P}$ the set of positive
integers and, for $n\in \mathbb{N}$, $[n]:=\Set{1,2,...,n}$ and $[\pm n]=\{-n,-n+1,...,-2,-1,1,2...,n\}$. We write $|X|$ for the cardinality of a set $X$,
$\subset$ for the proper inclusion between two sets and $\uplus$ for the disjoint union. For any function $f :X \rightarrow Y$ between two sets $X$ and $Y$ let $\imag(f):=\{f(x):x\in X\}$.

Let $(W,S)$ be a Coxeter system. By $\ell(z)$ we denote the length of the
element $z\in W$ with respect to the generator set $S$. If $J\subseteq S$, we define
\begin{align*} W^J&:=\Set{w\in W:\ell(ws)>\ell(w)~\forall~s\in J},
\\ D_L(w)&:=\Set{s\in S:\ell(sw)<\ell(w)},
\\ D_R(w)&:=\Set{s\in S:\ell(ws)<\ell(w)}.
\end{align*} The parabolic subgroup $W_J \subseteq W$ is the subgroup of $W$ generated by
$J\subseteq S$. In particular $W_S=W$ and
$W_\varnothing = \Set{e}$, being $e$ the identity of the group. The length $\ell_J$ on $(W_J,J)$ is equal to the restriction of $\ell$ at $W_J$, for all $J\subseteq S$, and so
it will be usually denoted by $\ell$ as well.

We are interested in the group $W$ as a graded poset with rank function $\ell$. The Bruhat order $\leqslant$, induced by its Coxeter presentation $(W,S)$, makes $W$ the desired poset
(see \cite[Chapter 2]{bjornerbrenti}).
If $W$ is finite, there exists a unique maximal
element $w_0$, which is the
element of maximal length in $W$. In this case we have that the Poincar\'e polynomial $W(q)$ is reciprocal (see \cite[Proposition 2.3.2]{bjornerbrenti}), i.e.
$$q^{\ell(w_0)}\sum\limits_{w\in W} q^{-\ell(w)}=\sum\limits_{w\in W} q^{\ell(w)}.$$

As it is well known, the set $W^J$ with the induced Bruhat order is graded by the length function, as stated in the following theorem
(see \cite[Theorem 2.5.5]{bjornerbrenti}).
\begin{thm} \label{WJgraduato} If $u<v$ in $W^J$, then there exist elements $v_i\in
W^J$, such that $\ell(v_i)=\ell(u)+i$, for $0 \leqslant i \leqslant
k$, and $u=v_0<v_1<...<v_k=v$.
\end{thm} 
For $J\subseteq S$, an element $w\in W$ has a unique expression
$w=w^Jw_J$, where $w^J\in W^J$, $w_J\in W_J$ and $\ell(w^J)+\ell(w_J)=\ell(w)$ (see
\cite[Proposition 2.4.4]{bjornerbrenti}). Moreover, as sets, $W\simeq W^J \times W_J$. This implies the following result.
\begin{lem} \label{lemDR}
  Let $(W,S)$ be a Coxeter system, $v\in W$ and $J\subseteq S$. Then $D_R(v_J)=D_R(v) \cap J$.
\end{lem}

The canonical projection $P^J:W \rightarrow W^J$,
defined by
\begin{equation*} P^J(w)=w^J,
\end{equation*} for all $w\in W$, is a morphism of
posets (see \cite[Proposition 2.5.1]{bjornerbrenti}), while the map $w \mapsto w_J$ is not. Therefore $W\simeq W^J \times W_J$ as sets but not as posets
(where the cartesian product of two posets is the cartesian product of the two sets with the product order: $(u,v)\leqslant (u',v') \Leftrightarrow u\leqslant u'$ and $v\leqslant v'$).

\begin{ex}
   Let us consider the group $S_3$. The factorization $w^Jw_J$ gives a bijection between the set of permutations $S_3$ generated by the simple inversions $s,t$ and the cartesian product $\{e,t,st\} \times \{e,s\}$, being $J=\{s\}$. The projection $P^{\{s\}}$ over the first factor  is a poset morphism. Nevertheless $S_3$ with the Bruhat order is not the cartesian product of the posets $\{e,t,st\}$ and $\{e,s\}$, both ordered with the induced Bruhat order.
\end{ex}

The set of reflections of a Coxeter system $(W,S)$ is $T:=\{wsw^{-1}:w\in W, s\in S\}$. For any $v\in W$, define $T(v):=\{t\in T:vt<v\}$.
It is known that $\ell(v)=|T(v)|$ (see \cite[Corollary 1.4.5]{bjornerbrenti}) and we can prove a more general result than the one stated in Lemma \ref{lemDR}.

\begin{lem} \label{TvJ} Let $(W,S)$ be a Coxeter system and $J\subseteq S$. Then
  $$T(v_J)=T(v)\cap W_J,$$ for all $v\in W$.
\end{lem}
\begin{proof} We have that either $w<wt$ or $wt<w$, for all $w\in W$, $t\in T$. If $t\in T(v_J)$ then clearly $t\in W_J$. But $v_J>v_Jt \in W_J$ implies $vt=v^Jv_Jt<v^Jv_J=v$, so $t\in T(v)$.
  If $t\in T(v) \cap W_J$ and $t\not \in T(v_J)$ then $v_J<v_Jt \in W_J$ and so $vt=v^Jv_Jt>v$, a contradiction.
\end{proof}

Given a sequence $\sigma = (a_1,a_2,...,a_n) \in \mathbb{Z}^n$, an \emph{inversion} of $\sigma$ is a pair $(i,j)\in [n] \times [n]$ such that $i<j$ and
$a_i>a_j$. We denote by $\sigma(k)$ the integer $a_k$, for all $k\in [n]$. Define $\negg(\sigma):=\{i\in [n]:\sigma(i)<0\}$. An element $i \in [n]$ is a \emph{descent} of $\sigma$ if $(i,i+1)$ is an inversion. The number of inversions of $\sigma$ will be denoted as $\inv(\sigma)$, the set of descents by $D(\sigma)$ and its
\emph{major index} is defined by $$\maj(\sigma):= \sum \limits_{i \in D(\sigma)} i.$$

\section{Induced functions, involutions and symmetric pairs}
\label{sec2}

Let $(X,\leqslant,\rho)$ be a finite graded poset with maximum $\hat{1}$, minimum $\hat{0}$ and rank function $\rho$. Let $F(X)$ be the ring of functions $f: X \rightarrow \mathbb{Z}$. Two elements $f,g \in F(X)$ are
\emph{equidistributed} if the following equality holds in the semiring of Laurent polynomials $\mathbb{N}[q,q^{-1}]$:

$$\sum \limits_{x\in X}q^{f(x)} =\sum \limits_{x\in X}q^{g(x)}.$$

On $F(X)$ define the following equivalence relation: $f \sim g$ if and only if they satisfy these conditions:
\begin{enumerate}
  \item $f$ and $g$ are equidistributed;
  \item $f(\hat{0})=g(\hat{0})$ and
  \item  $f(\hat{1})=g(\hat{1})$.
\end{enumerate}
 The equivalence class of a function
$f\in F(X)$ is denoted by $[f]$. Note that if $f\in [\rho]$ then $f(\hat{0})=0$ and $f(x)>0$ for all $x\in X\setminus\{\hat{0}\}$.

Let us define the functions $k^+, k^- : F(X) \times F(X)\rightarrow \mathbb{N}\cup \{-1\}$ by

$$k^+(f,g):=|\{f(x)+g(y):x,y\in X\}|-|\mathrm{Im}(f+g)|-1,$$
$$k^-(f,g):=|\{f(x)-g(y):x,y\in X\}|-|\mathrm{Im}(f-g)|-1,$$
for all $f,g \in F(X)$.
  Clearly $k^+(f,g)=k^+(g,f)$ and $k^-(f,g)=k^-(g,f)$ for all $f,g \in F(X)$.
It is straightforward to see that $f\in [\rho]$ implies the equality \begin{equation}\label{f+r}\{\rho(x)+f(y):x,y\in X\}=\{0,1,...,2\rho(\hat{1})\}.\end{equation}

\begin{lem} \label{lemma+}
  The inequality
        $$k^+(\rho,f) \geqslant 1$$
 holds for all $f\in [\rho]$. Moreover $k^+(\rho,f) = 1$ if and only if $\imag(\rho+f)=\{0,2,3,...,2\rho(\hat{1})-2,2\rho(\hat{1})\}$.
\end{lem}
\begin{proof}
  Let $f\in [\rho]$;  then $f(x)\in \{0,\rho(\hat{1})\}$ if and only if $x\in \{\hat{0},\hat{1}\}$. Therefore $\{1,2\rho(\hat{1})-1\} \cap \imag(\rho+f)=\varnothing$, since $1=1+0$ and $2\rho(\hat{1})-1=\rho(\hat{1})+\rho(\hat{1})-1$ are the only acceptable compositions of these numbers.
\end{proof}

As a direct consequence of \eqref{f+r}  we find the following result.
\begin{prop} \label{prop+}
  Let $f \in [\rho]$. Then
$$|\mathrm{Im}(\rho+ f)|=2\rho(\hat{1})-k^+(\rho,f).$$
\end{prop}

For an analogous result relative to $\imag(\rho-f)$ we need one more condition. A polynomial $P\in \mathbb{Z}[q]$ of degree $k$ is reciprocal if
$q^kP(q^{-1})=P(q)$. Then in the case $\sum \limits_{x\in X}q^{\rho(x)}$ is a reciprocal polynomial we obtain that $\{f(x)+\rho(\hat{1})-\rho(y):x,y\in X\}=\{0,1,...,2\rho(\hat{1})\}$ and then \begin{equation}\label{f-r}\{\rho(x)-f(y):x,y\in X\}=\{-\rho(\hat{1}),...,-1,0,1,...,\rho(\hat{1})\}.\end{equation}
As a consequence of \eqref{f-r}, we can state the next result.

\begin{prop} \label{prop-}
  Let $f \in [\rho]$. If $\sum \limits_{x\in X}q^{\rho(x)}$ is a reciprocal polynomial then
$$|\mathrm{Im}(\rho- f)|=2\rho(\hat{1})-k^-(\rho,f).$$
\end{prop}

Since $|\mathrm{Im}(\rho+\rho)|=|\mathrm{Im}(\rho)|=\rho(\hat{1})+1$ and $|\mathrm{Im}(\rho-\rho)|=1$,
by Propositions \ref{prop+} and \ref{prop-} we obtain $k^+(\rho,\rho)=\rho(\hat{1})-1$ and $k^-(\rho,\rho)=2\rho(\hat{1})-1$, when the polynomial $\sum \limits_{x\in X}q^{\rho(x)}$ is reciprocal.

Let $X=A \times B$ be the cartesian product of the sets $A$ and $B$. For any $a\in  A$ and $b\in B$ we have the inclusions of sets $i^b_A:A \rightarrow X$, $i^a_B :B \rightarrow X$ satisfying $\imag(i^b_A)=A \times \{b\}$ and $\imag(i^a_B)=\{a\} \times B$. So the order $\leqslant$ on $X$ induces an order on $A$ and an order on $B$, for which such inclusions are poset morphisms.
The order on $X$ is a refinement of the product order of $A \times B$; the two projections $\pi_A:X \rightarrow A$ and $\pi_B:X \rightarrow B$
are not poset morphisms in general. We have that $\hat{0}=(\hat{0}_A,\hat{0}_B)$ and
$\hat{1}=(\hat{1}_A,\hat{1}_B)$.

Given a function $f\in F(X)$, one can define the functions $f_A\in F(A)$ and $f_B\in F(B)$ by
$f_A(a)=f(a,\hat{0}_B)$ and $f_B(b)=f(\hat{0}_A,b)$, for all $a\in A$, $b\in B$.

\begin{dfn}
  Let $(X,\leqslant,\rho)$ be a graded poset with rank function $\rho$. We call a {decomposition}\footnote{To be precise we should write $X\simeq A\times B$, since we only need a bijection between the sets $X$ and $A\times B$; the use of the equality avoids notational complications and it is not restrictive.} $X=A\times B$ \emph{good} if $(A,\leqslant,\rho_A)$ and $(B,\leqslant,\rho_B)$ are graded posets and $\rho=\rho_A\circ \pi_A+\rho_B \circ \pi_B$.
\end{dfn}

\begin{oss} \label{oss prodotto cartesiano}
  The cartesian product of graded posets $P \times Q$ clearly admits a good decomposition.
\end{oss}

For a graded poset $(X,\leqslant,\rho)$ with a good decomposition $X=A\times B$, define the function $R:F(X)\rightarrow F(X)$ by $$R(f)=f-\rho_A \circ \pi_A,$$
for all $f\in F(X)$. Clearly $R(f)(\hat{0}_A,b)=f_B(b)$ for all $b\in B$ and $R^k(f)=f-k\rho_A \circ \pi_A$, for all $k\in \mathbb{Z}$.

\begin{prop} \label{propind}
  Let $X=A\times B$ be a good decomposition of the poset $X$, $f\in F(X)$ and $g\in F(B)$. If $g\in [\rho_B]$ and $R(f)=g\circ \pi_B$, then $f \in [\rho]$.
\end{prop}
\begin{proof} By definition $f=R(f)+\rho_A \circ \pi_A$. Then
  \begin{eqnarray*}
    \sum_{x\in X}q^{f(x)} &=&  \sum_{(a,b) \in A \times B}q^{R(f)(a,b)+\rho_A \circ \pi_A(a,b)} = \sum_{a \in A}q^{\rho_A(a)} \sum _{b\in B}q^{g(b)} \\
    &=& \sum_{a \in A}q^{\rho_A(a)} \sum _{b\in B}q^{\rho_B(b)} = \sum_{(a,b)\in A \times B}q^{\rho_A(a)+\rho_B(b)} \\&=& \sum_{x\in X}q^{\rho(x)}.
  \end{eqnarray*} Moreover $f(\hat{0})=g(\hat{0}_B)=0=\rho(\hat{0})$ and $f(\hat{1})=\rho_{A}(\hat{1}_A)+g(\hat{1}_B)=\rho_{A}(\hat{1}_A)+\rho_{B}(\hat{1}_B)=\rho(\hat{1})$.
\end{proof} In view of Proposition \ref{propind}, the following definition seams natural.
\begin{dfn} \label{funzione indotta}
  Let $X$ be a finite graded poset with rank function $\rho$ and a good decomposition $A\times B$. We say that a function $f \in [\rho]$  is \emph{induced} by a function $g \in F(B)$ (equivalently that $g$ \emph{induces} $f$)
if $g\in [\rho_B]$ and $R(f)=g \circ \pi_B$.
\end{dfn} Now we can
establish two results about the image of $\rho+f$ and $\rho-f$, when $f\in F(X)$ is induced by a function $g\in F(B)$.

\begin{thm} \label{teorema+}
  Let $f \in [\rho]$ be induced by $g\in [\rho_B]$. If $k^+(\rho_B,g)=1$ then
$k^+(\rho,f)=1$.
\end{thm}
\begin{proof}
  If $\rho(\hat{1}_B)\leqslant 2$ then $g=\rho_B$ and $f=\rho$, so the result follows. Let $\rho(\hat{1}_B)>2$. Since $k^+(\rho_B,g)=1$, by Lemma \ref{lemma+} we have that $\imag(\rho_B+g)=\{0,2,3,...,2\rho(\hat{1}_B)-2,2\rho(\hat{1}_B)\}$. Let $b,b'\in B$ be elements satisfying $\rho_B(b)+g(b)=2\rho(\hat{1}_B)-2$ and
 $\rho_B(b')+g(b')=2\rho(\hat{1}_B)-3$. Let $a_0=\hat{0}_A < a_1 <...< a_n =\hat{1}_A$ be a maximal chain in $A$. Then $f(a_i,b)=\rho_A(a_i)+g(b)=i+g(b)$, for all $0\leqslant i \leqslant n$. Therefore \begin{eqnarray*}
                                I&:=&\{0,2,3,...,\rho_B(b')+g(b'),\rho_B(b)+g(b)\} \\&&
\uplus  ~\{\rho_B(b)+g(b)+2,...,\rho_B(b)+g(b)+2\rho(\hat{1}_A)\} \\
                                 && \uplus ~ \{\rho_B(b')+g(b')+2,...,\rho_B(b')+g(b')+2\rho(\hat{1}_A)\} \\
                                 && \uplus ~\{2\rho(\hat{1})\} \\
&=& \{0,2,3,...,\rho_B(b')+g(b'),\rho_B(b)+g(b)\}
\\&& \uplus  ~\{\rho(a_1,b)+f(a_1,b),...,\rho(a_n,b)+f(a_n,b)\} \\
                                 && \uplus ~\{\rho(a_1,b')+f(a_1,b'),...,\rho(a_n,b')+f(a_n,b')\} \\
                                 && \uplus ~\{\rho(a_n,\hat{1}_B)+f(a_n,\hat{1}_B)\}
                              \end{eqnarray*} and $I\subseteq \imag(\rho+f)$. Hence, by Lemma \ref{lemma+}, $2\rho_B(\hat{1}_B)+2\rho_A(\hat{1}_A)-1\leqslant |I|\leqslant 2\rho(\hat{1})-1$ and the result follows.
\end{proof}

For $\imag(\rho-f)$ we can prove a stronger result, which can be used as a criterium to establish if a function is not induced by another.

\begin{thm} \label{teorema-}
  Let $f \in [\rho]$ be induced by $g\in [\rho_B]$. Then
$$\imag(\rho-f)=\imag(\rho_B-g).$$ In particular, if $\sum \limits_{x\in X}q^{\rho(x)}$ and $\sum \limits_{x\in X}q^{\rho_B(x)}$ are reciprocal polynomials then $$k^-(\rho,f)=2\rho_A(\hat{1}_A)+k^-(\rho_B,g).$$
\end{thm}
\begin{proof}
 Since $f$ is induced by $g$, by definition $f=\rho_A\circ \pi_A+R(f)= \rho_A\circ \pi_A+g \circ \pi_B$. Then $\rho-f=\rho-\rho_A\circ \pi_A-g \circ \pi_B  =\rho_B \circ \pi_B-g \circ \pi_B$.
But $\imag(\rho_B \circ \pi_B-g \circ \pi_B)=\imag(\rho_B-g)$. The last assertion is implied by Proposition \ref{prop-}.
 \end{proof}

For two equidistributed functions $f,g$ over a finite set $X$, the following proposition gives an easy criterium to exclude the existence of an involution $\iota : X \rightarrow X$ such that $f=g\circ \iota$. The proof is immediate.

\begin{prop} \label{criterioinv}
  Let $X$ be a finite set and $f,g\in F(X)$. If there exists an involution $\iota : X \rightarrow X$  such that $f=g\circ \iota$, then $f$ and $g$ are equidistributed and $$\sum \limits_{\substack{x\in X, \\g(x)\neq 0}} \frac{f(x)}{g(x)} = \sum \limits_{\substack{x\in X, \\f(x)\neq 0}} \frac{g(x)}{f(x)}.$$
\end{prop}

When we have a good decomposition $X=A \times B$ and a function $f\in [\rho]$ induced by a function $g\in [\rho_B]$ we can prove the following theorem.

\begin{thm} \label{invol}
  Let $f\in [\rho]$ be induced by a function $g\in [\rho_B]$ and $\iota : B \rightarrow B$ an involution such
that $g=\rho_B\circ \iota$. Then the function $\tilde{\iota}: X \rightarrow X$ defined by
$$\tilde{\iota}(a,b):=(a,\iota(b)),$$
for all $(a,b)\in A\times B$, is an involution and $f=\rho\circ \tilde{\iota}$.
\end{thm}
\begin{proof}
  The fact that $\tilde{\iota}$ is an involution is evident. Moreover,
  for $x=(a,b)$ we have $f(x) = \rho_A(a)+g(b) =\rho_A(a)+\rho_B(\iota(b))=\rho(a,\iota(b))=\rho(\tilde{\iota}(x))$.
\end{proof}

The last results which we want to state in this generality concern symmetry in bivariate distributions.

\begin{dfn}\label{def simmetria} A pair $(f,g) \in F(X) \times F(X)$ is said \emph{symmetric} if  the following equality holds in the semiring of Laurent polynomials  $\mathbb{N}[q,t,q^{-1},t^{-1}]$:

$$\sum \limits_{x\in X} q^{f(x)}t^{g(x)}=\sum \limits_{x\in X} q^{g(x)}t^{f(x)}.$$  \end{dfn} Clearly if $(f,g)$ is a symmetric pair then $f$ and $g$ are equidistributed.

The next theorem asserts that given a good decomposition $X = A\times B$ of a finite graded poset $(X,\leqslant,\rho)$, a symmetric pair $(g,\rho_B)$ induces a symmetric pair $(f,\rho)$.

\begin{thm} \label{simmetria}
  Let $f\in [\rho]$ be induced by $g\in [\rho_B]$. If $(g,\rho_B)$ is symmetric then $(f,\rho)$ is symmetric.
\end{thm}
\begin{proof}
  If $f\in [\rho]$ is induced by $g\in [\rho_B]$ then $f=\rho_A \circ \pi_A +g\circ \pi_B$. Therefore
  \begin{eqnarray*}
    \sum \limits_{x\in X} q^{\rho(x)}t^{f(x)} &=& \sum \limits_{a\in A} \sum \limits_{b\in B} q^{\rho_A(a)+\rho_B(b)}t^{\rho_A(a) +g(b)} \\
     &=& \sum \limits_{a\in A}q^{\rho_A(a)}t^{\rho_A(a)} \sum \limits_{b\in B} q^{\rho_B(b)}t^{g(b)} \\
     &=& \sum \limits_{a\in A}q^{\rho_A(a)}t^{\rho_A(a)} \sum \limits_{b\in B} t^{\rho_B(b)}q^{g(b)} \\
     &=& \sum \limits_{x\in X} q^{f(x)}t^{\rho(x)}.
  \end{eqnarray*}
\end{proof}

The existence of an involution $\iota : X \rightarrow X$ such that $f=g\circ \iota$  implies the symmetry of the pair $(f,g)$.
In fact
$$\sum \limits_{x\in X}q^{f(x)}t^{g(x)}=\sum \limits_{x\in X}q^{g(\iota(x))}t^{f(\iota(x))}=\sum \limits_{x\in X}q^{g(x)}t^{f(x)}.$$
We have also that symmetry implies the existence of such an involution, so Theorems  \ref{invol} and \ref{simmetria} are equivalent, as we state in the next theorem.

\begin{thm} \label{thm symm}
  The pair $(f,g)$ is symmetric if and only if there exists an involution $\iota: X \rightarrow X$ such that $f=g\circ \iota$.
\end{thm}
\begin{proof} We have already discussed one direction. So let $(f,g)$ be symmetric.
  Then $f$ and $g$ are equidistributed and $|\{x\in X:(f(x),g(x))=(h,k)\}|=|\{x\in X:(f(x),g(x))=(k,h)\}|$ for all
  $h,k \in \imag(f)=\imag(g)$. Any function $\iota$ which fixes the elements of the set $\{x\in X:f(x)=g(x)\}$ and which matches the elements of $\{x\in X:(f(x),g(x))=(h,k)\}$ with the elements of $\{x\in X:(f(x),g(x))=(k,h)\}$ is an involution which satisfies $f=g\circ \iota$.
\end{proof}

Let $X=\{\hat{0},x_1,...,x_6,\hat{1}\}$ be a chain of eigth elements and consider the function $f=(0,3,1,6,5,4,2,7)$ (which means that $f(\hat{0})=0$, $f(x_1)=3$, etc.); then the rank of $X$ is $\rho=(0,1,2,3,4,5,6,7)$ and $f\in [\rho]$. One can check that the equality of Proposition \ref{criterioinv} is satisfied and that $(f,\rho)$ is not symmetric.

The definition of good decomposition of a finite graded poset was formulated having in mind Coxeter systems and parabolic subgroups. For a finite Coxeter system $(W,S)$
and a set $J\subseteq S$, by the factorization $w=w^Jw_J$ given in Section \ref{sec1} and by Theorem \ref{WJgraduato}, the Bruhat order on $W$ realizes a good decomposition $W=W^J \times W_J$. Then a function $g\in F(W_J)$ such that $g\in [\ell_J]$  induces a function $f\in F(W)$ defined by

\begin{equation}\label{f elle g}
  f(w)=\ell(w^J)+g(w_J),
\end{equation} for all $w\in W$, which, by Proposition \ref{propind}, lies in the class $[\ell]$.
The function $f$ coincides with $g$ on $W_J$ and with $\ell$ on $W^J$. The other values are taken on the elements of the set $W \setminus (W^J \cup W_J)$, whose
cardinality  is $$|W \setminus (W^J \cup W_J)|=|W|-|W_J|-|W|/|W_J|+1,$$ since $|W^J|=|W|/|W_J|$ and $W^J\cap W_J=\{e\}$ for all $J\subseteq S$.

  We have to stress now that there is a left version of our results, considering the set ${^J}W:=\Set{w\in W:\ell(sw)>\ell(w)~\forall~s\in J}$ instead of $W^J$.
  This left version generates other functions equidistributed with the length. If $f \in [\ell]$ is induced by $g\in [\ell_J]$ on the right
  then the function $f^*$ is the one induced by $g$ on the left, where we have defined the involution $*$ on $F(W)$ by $$f^*(w):=f(w^{-1}),$$ for all $w\in W$.
  Observe that this involution fixes the equivalence classes $[f]$ of $F(W)$, since $e$ and $w_0$ are involutions.
  The function $f^*$ is induced by
  $g^*\in [\ell_J]$ on the right.
   Moreover we have  that
  $\imag(\ell-f^*)=\imag(\ell_J-g)=\imag(\ell_J-g^*)=\imag(\ell-f)$.

 \section{Type A}

We consider now, for $n\in \mathbb{P}$, the Coxeter system $(S_{n+1},S)$ of type $A_n$.
The group $S_{n+1}$ of permutations of $n+1$ objects is generated by the transpositions $S=\{s_1,s_2,...,s_n\}$ which are involutions satisfying the relations

$$s_is_js_i=\left\{
              \begin{array}{ll}
                s_js_is_j, & \hbox{if $|i-j|=1$;} \\
                s_j, & \hbox{if $|i-j|>1$.}
              \end{array}
            \right.
$$ Here and in the subsequent sections we denote by $\ell$ the length function relative to this presentation. In this section we identify the elements of $S_{n+1}$ with sequences of $n+1$ numbers in $[n+1]$. If we represent the identity as the sequence $e=(1,2,...,n+1)$ and if the generators act in the right on a sequence $(\sigma(1),...,\sigma(n+1)) \in S_{n+1}$  interchanging
$\sigma(i)$ with $\sigma(i+1)$, and in the left, interchanging $i$ with $i+1$, we have that $s_i=(1,...,i-1,i+1,i,i+2,...,n+1)$ for all $i\in [n]$.
With this identification, given a sequence $\sigma \in S_{n+1}$, its length $\ell(\sigma)$ corresponds to its number of
inversions $\inv(\sigma)=|T(\sigma)|$ and $D_R(\sigma)=D(\sigma)$. The maximal element $w_0$ corresponds to the sequence $(n+1,n,n-1,...,1)$ and $\ell(w_0)=\frac{n(n+1)}{2}$. See \cite{bjornerbrenti} for further details.

The \emph{major index} of $\sigma$ is the function $\maj \in F(S_{n+1})$ which is defined at the end of Section \ref{sec1}.
A classical result of MacMahon asserts that $\ell$ and $\maj$ are equidistributed (see \cite[Theorem 2.17]{bona}).
Moreover we know that the pair $(\maj,\ell)$ is symmetric over the group $S_{n+1}$  (see \cite{foata}), i.e
\begin{equation}\label{simmetria maj}
  \sum \limits_{\sigma \in S_{n+1}} x^{\ell(\sigma)}y^{\maj(\sigma)}= \sum \limits_{\sigma \in S_{n+1}} x^{\maj(\sigma)}y^{\ell(\sigma)}.
\end{equation}
This fact implies, by Theorem \ref{thm symm}, the existence of an involution $\iota: S_{n+1} \rightarrow S_{n+1}$ such that $\maj = \ell \circ \iota$. So, by Proposition \ref{criterioinv}, we can deduce the following equality, for all $n \in \mathbb{P}$:

$$\sum \limits_{\substack{\sigma \in S_{n+1}, \\\sigma \neq e}} \frac{\ell(\sigma)}{\maj(\sigma)} = \sum \limits_{\substack{\sigma \in S_{n+1}, \\\sigma \neq e}} \frac{\maj(\sigma)}{\ell(\sigma)}.$$

By Propositions \ref{prop+} and \ref{prop-}
we have that $|\imag(\ell+\maj)|=n(n+1)-k^+(\ell,\maj)$ and $|\imag(\ell-\maj)|=n(n+1)-k^-(\ell,\maj)$. Moreover, Lemma \ref{lemma+} implies that $$|\imag(\ell+\maj)|\leqslant n(n+1)-1.$$   The following theorem states the exact
value of $|\imag(\ell+\maj)|$.

\begin{thm} \label{teoremaA}
  Let $n\in \mathbb{P}$ and $\ell, \maj \in F(S_{n+1})$ be the functions length and major index. Then
  $$|\imag(\ell+\maj)|= \left\{
                              \begin{array}{ll}
                                2, & \hbox{if $n=1$;} \\
                                n(n+1)-1, & \hbox{otherwise.}
                              \end{array}
                            \right.
  $$
\end{thm}
\begin{proof}
If $n=1$ the result follows by a straightforward calculation. Let $n>1$ and, for $0\leqslant j \leqslant n-1$ and $i \in [n-j]$, define the sequence $\sigma_{i,j} \in S_{n+1}$ by

$$\sigma_{i,j}(k)=\left\{
    \begin{array}{ll}
      k, & \hbox{if $k> n+1-j$;} \\
      n+2-j-k, & \hbox{if $n+1-j\geqslant k>i$;} \\
      n+1-j, & \hbox{if $k=i$;} \\
      n+1-j-k, & \hbox{if $k<i$,}
    \end{array}
\right.
$$
for all $k\in [n+1]$. For example, in $S_5$ we have $\sigma_{1,0}=(5,4,3,2,1)$, $\sigma_{2,2}= (2,3,1,4,5)$ and $\sigma_{1,3}= (2,1,3,4,5)$.

Let $U_n:=\{\sigma_{i,j}:0\leqslant j \leqslant n-1, ~i \in [n-j]\} \subseteq S_{n+1}$. The map $$\phi: U_n \setminus \{\sigma_{1,n-1}\} \rightarrow U_n \setminus \{w_0\}$$ defined by $\phi(\sigma_{i,j})=\sigma_{i,j}s_i$ is a bijection. In fact
we have that, if $j\leqslant n-1$, $$\sigma_{i,j}s_i= \left\{
                                     \begin{array}{ll}
                                       \sigma_{i+1,j}, & \hbox{if $i<n-j$;} \\
                                       \sigma_{1,j+1}, & \hbox{if $i=n-j$,}
                                     \end{array}
                                   \right.$$ for all $i \in [n-j]$. Clearly $\ell(\phi(\sigma_{i,j}))=\ell(\sigma_{i,j})-1$ and one can see that $\maj(\phi(\sigma_{i,j}))=\maj(\sigma_{i,j})-1$. Therefore $$|\{\ell(\sigma)+\maj(\sigma):\sigma \in U_n\}|=|U_n|=\frac{n(n+1)}{2}=\ell(w_0).$$

Let us define a map $\psi: U_n \setminus \{\sigma_{1,0},\sigma_{2,0}\} \rightarrow S_{n+1}$ by
$$\psi(\sigma_{i,j})=\left\{
                       \begin{array}{ll}
                         s_n\sigma_{i,j}, & \hbox{if $j=0$;} \\
                         s_{n+1-j}\sigma_{i,j}, & \hbox{otherwise,}
                       \end{array}
                     \right.$$ for all $(i,j) \not \in \{(1,0),(2,0)\}$ and let $J_n:=\imag(\psi)$.

It is clear that $\ell(\psi(\sigma_{i,j}))=\ell(\sigma_{i,j})+1$. Moreover $\maj(\psi(\sigma_{i,j}))=\maj(\sigma_{i,j})$ and
$J_n \cap U_n = \varnothing$. We have that $(\ell+\maj)(\sigma) \equiv 0 \mod 2$ for all $\sigma \in U_n$ and $(\ell+\maj)(\sigma) \equiv 1 \mod 2$
for all $\sigma \in J_n$. Hence
\begin{eqnarray*}
  n(n+1)-1 &\geqslant&  |\imag(\ell+\maj)| \geqslant |U_n \cup J_n \cup \{e\}| \\
  &=& |U_n|+|J_n|+1=|U_n|+|U_n|-1 \\&=&2\ell(w_0)-1=n(n+1)-1,
\end{eqnarray*}  and the result follows.

\end{proof}

By Theorem \ref{teoremaA} and Proposition \ref{prop+} we obtain the following corollary.

\begin{cor} Let $n\in \mathbb{P}$ and $\ell, \maj \in F(S_{n+1})$ be the functions length and major index. Then
$k^+(\ell,\maj)=1$ for all $n>1$.
\end{cor}

Let consider the function $\ell-\maj$. For $1\leqslant n \leqslant 11$ a computer calculation yields

\begin{equation}\label{successione-}
  |\imag(\ell-\maj)|=1, 3, 5, 9, 15, 21, 29, 39, 49, 51, 63
\end{equation} and then $k^-(\ell,\maj)=1,3,7,11,15,21,27,33,41,59,69$.

\begin{ex}
  Let us compute the function $f\in F(S_4)$ induced by $\maj\in F(S_3)$ injecting $(S_3,\{s_1,s_2\})$ in $(S_4,\{s_1,s_2,s_3\})$. So we are identifying $S_3$ with the set of elements $u\in S_4$ such that $u(4)=4$.  We have that $f(u)=\maj(u)$ for all $u\in S_4$ such that $u(4)=4$ and $f(u)=\ell(u)$ for all $u\in S_4^{\{s_1,s_2\}}$. On the set $S_4 \setminus \left(S_4^{\{s_1,s_2\}} \cup S_3 \right)$ we obtain
   \begin{itemize}
     \item $f(1,4,2,3)=\ell(1,2,4,3)+\maj(1,3,2)=1+2=3$,
     \item $f(2,1,4,3)=\ell(1,2,4,3)+\maj(2,1,3)=1+1=2$,
     \item $f(4,1,2,3)=\ell(1,2,4,3)+\maj(3,1,2)=1+1=2$,
     \item $f(2,4,1,3)=\ell(1,2,4,3)+\maj(2,3,1)=1+2=3$,
     \item $f(4,2,1,3)=\ell(1,2,4,3)+\maj(3,2,1)=1+3=4$,
     \item $f(1,4,3,2)=\ell(1,3,4,2)+\maj(1,3,2)=2+2=4$,
     \item $f(3,1,4,2)=\ell(1,3,4,2)+\maj(2,1,3)=2+1=3$,
     \item $f(4,1,3,2)=\ell(1,3,4,2)+\maj(3,1,2)=2+1=3$,
     \item $f(3,4,1,2)=\ell(1,3,4,2)+\maj(2,3,1)=2+2=4$,
     \item $f(4,3,1,2)=\ell(1,3,4,2)+\maj(3,2,1)=2+3=5$,
     \item $f(2,4,3,1)=\ell(2,3,4,1)+\maj(1,3,2)=3+2=5$,
     \item $f(3,2,4,1)=\ell(2,3,4,1)+\maj(2,1,3)=3+1=4$,
     \item $f(4,2,3,1)=\ell(2,3,4,1)+\maj(3,1,2)=3+1=4$,
     \item $f(3,4,2,1)=\ell(2,3,4,1)+\maj(2,3,1)=3+2=5$,
     \item $f(4,3,2,1)=\ell(2,3,4,1)+\maj(3,2,1)=3+3=6$.
   \end{itemize} As a consequence of our results we have that $f\in [\ell]$ on $S_4$. By Theorem \ref{simmetria} and the symmetry of $(\maj,\ell)$ on $S_3$, we can deduce that the pair $(f,\ell)$ is  symmetric.
\end{ex}

We end this section by proving that the functions $\ell+\maj$ and $\imaj+\maj$ are equidistributed, as the functions
$\ell-\maj$ and $\imaj-\maj$. Let $D_I:=\{\sigma \in S_{n+1}:D_R(\sigma)=I\}$. As proved in \cite{foata} the length function and the inverse major index are equidistributed
over $D_I$, for all $I \subseteq S$, i.e.

\begin{equation} \label{ellimaj}
 \sum \limits_{\sigma \in D_I} x^{\ell(\sigma)}=\sum \limits_{\sigma \in D_I} x^{\maj^*(\sigma)}.
\end{equation}

This fact implies the following equality in the polynomial semiring $\mathbb{N}[x,y]$:

\begin{equation} \label{eqmajimaj}
   \sum \limits_{\sigma \in S_{n+1}} x^{\ell(\sigma)}y^{\maj(\sigma)}=\sum \limits_{\sigma \in S_{n+1}} x^{\maj^*(\sigma)}y^{\maj(\sigma)}.
\end{equation}
This equality follows since the set $\{D_I:I\subseteq S\}$ constitutes a partition of $S_{n+1}$ and therefore, by Equation \eqref{ellimaj},

\begin{eqnarray*}
  \sum \limits_{\sigma \in S_{n+1}} x^{\ell(\sigma)}y^{\maj(\sigma)} &=& \sum \limits_{I\subseteq S}y^{\maj(\sigma)}\sum \limits_{\sigma \in D_I} x^{\ell(\sigma)} \\
   &=& \sum \limits_{I\subseteq S}y^{\maj(\sigma)}\sum \limits_{\sigma \in D_I} x^{\maj^*(\sigma)} \\
 &=& \sum \limits_{\sigma \in S_{n+1}} x^{\maj^*(\sigma)}y^{\maj(\sigma)}.
\end{eqnarray*}
Similar results are also discussed in \cite[Theorem 2.3 and Corollary 2.4]{AdinBrenti3}. Therefore we have that $\ell+ \maj\sim \imaj+ \maj$ and $\ell- \maj \sim \imaj- \maj$.
By the criterium enounced in Proposition \ref{criterioinv} one can see, by a direct calculation, that there is no involution $\iota: S_{n+1} \rightarrow S_{n+1}$
such that $\ell\pm \maj= (\imaj \pm \maj)\circ\iota$.

\section{Type B}
\label{sec4}

A Coxeter system of type $B_n$ can be realized with the group $S^B_n$ of bijections $w$ of the set $[\pm n]$ such that $w(-a)=-w(a)$ for all $a\in [n]$, generated by the set of involutions $S=\{s_0,s_1,...,s_{n-1}\}$
 where, as sequences, $$s_i=(1,...,i-1,i+1,i,i+2,...,n)$$ and $$s_0=(-1,2,...,n),$$ and which satisfy the relations

$$s_is_js_i=\left\{
              \begin{array}{ll}
                s_1s_0s_1s_0s_1, & \hbox{if $i=0$ and $j=1$;} \\
                s_js_is_j, & \hbox{if $|i-j|=1$ and $i,j\neq 0$;} \\
                s_j, & \hbox{if $|i-j|>1$.}
              \end{array}
            \right.
$$
We denote by $\ell_B$ the length function relative to this presentation.
The generator $s_0$ acts in the right of a sequence $(\sigma(1),...,\sigma(n))$ by interchanging $\sigma(1)$ with $-\sigma(1)$ and in the left
 by interchanging  $1$ with $-1$. The element $w_0$ is represented by the sequence $(-1,-2,...,-n)$. See \cite[Chapter 8]{bjornerbrenti} for more details. So we are identifying the elements of $S^B_n$ with sequences of $n$ integers subject to some conditions. With this identification, if $J=\{s_1,...,s_{n-1}\}$,
by Lemma \ref{lemDR} and by Lemma \ref{TvJ} we have that $D(\sigma)=D_R(\sigma_J)$ and $\inv(\sigma)=|T(\sigma_J)|=\ell_B(\sigma_J)=\ell(\sigma_J)$,
for all sequences $\sigma \in S^B_n$. Therefore, since (see \cite[Equality (8.3)]{bjornerbrenti})

\begin{eqnarray} \label{eqlB}
 \nonumber 
  \ell_B(\sigma) = \inv(\sigma)-\sum \limits_{i\in \negg(\sigma)}\sigma(i),
\end{eqnarray} we find that $\ell_B(\sigma^J)=-\sum \limits_{i\in \negg(\sigma)}\sigma(i)$, for all $\sigma \in S^B_n$, and $\ell_B(w_0)=n^2$.

In \cite{AdinBrenti}, the authors introduced a function $\nmaj \in F(S^B_n)$ equidistributed with $\ell_B$, the \emph{negative major index}.
This is defined by
\begin{eqnarray} \label{eqnmaj}
 \nonumber 
  \nmaj(\sigma) = \maj(\sigma)-\sum \limits_{i\in \negg(\sigma)}\sigma(i),
\end{eqnarray} for all $\sigma \in S^B_n$. Therefore $\nmaj=\maj+\ell_B \circ P^J$, where $P^J$ is the canonical projection defined in Section \ref{sec1}. The following proposition states that $\nmaj$ is induced by $\maj \in F(S_n)$, in the sense of Definition \ref{funzione indotta}.

\begin{prop} \label{prop-nmaj}
  The function $\nmaj\in F(S^B_n)$ is induced by $\maj \in F(S_n)$, for all $n\geqslant 2$.
\end{prop}
\begin{proof}
  By Lemma \ref{lemDR} we have that $\maj(\sigma)=\maj(\sigma_J)$ for all $\sigma \in S^B_n$. As we have already observed $\nmaj=\maj+\ell_B \circ P^J$,
 so the result follows.
\end{proof}

Since we known that there exists an involution $\iota : S_n \rightarrow S_n$ such that $\maj =\ell \circ \iota$, we can state the next result.

\begin{cor} \label{involB}
  There exists an involution $\tilde{\iota} : S_n^B \rightarrow S_n^B$ such that $$\nmaj=\ell_B \circ \tilde{\iota}.$$
  This involution satisfies $\tilde{\iota}(\sigma)=P^J(\sigma)\iota(\sigma_J)$, for all $\sigma \in S_n^B$.
\end{cor}
\begin{proof}
  The result follows by Proposition \ref{prop-nmaj} and Theorem \ref{invol}.
\end{proof}

As a direct consequence of Theorem \ref{simmetria} and Proposition \ref{prop-nmaj} we obtain the result of \cite[Proposition 5.2]{Biagioli2}.
\begin{cor}
  The pair $(\nmaj,\ell_B)$ is symmetric.
\end{cor}

Knowing that $\nmaj\in F(S^B_n)$ is induced by $\maj \in F(S_n)$ and the behaviour of $\imag(\ell+\maj)$ and $\imag(\ell-\maj)$ on $S_n$,
we can deduce the behaviour of $\imag(\ell_B+\nmaj)$ and $\imag(\ell_B-\nmaj)$.

\begin{prop} \label{propB}
   Let $n>2$ and $\ell_B, \nmaj \in F(S^B_n)$ be the functions length and negative major index. Then
  $$|\imag(\ell_B+\nmaj)|= \left\{
                              \begin{array}{ll}
                                5, & \hbox{if $n=2$;} \\
                                2n^2-1, & \hbox{otherwise,}
                              \end{array}
                            \right.
  $$ and $\imag(\ell_B-\nmaj)=\imag(\ell-\maj)$, being $\maj \in F(S_n)$ the major index.
\end{prop}
\begin{proof}
  The case $n=2$ is obtained by a direct computation. Let $n>2$. Since $\nmaj$ is induced by $\maj$ and $\ell_B(w_0)=n^2$, the result follows by Theorems \ref{teorema+}, \ref{teorema-} and \ref{teoremaA}.
\end{proof}

The function  \emph{flag major index} $\fmaj\in F(S^B_n)$ introduced in \cite{AdinBrenti} and defined by
\begin{eqnarray} \label{fmaj}
 \nonumber 
  \fmaj(\sigma) = 2\maj(\sigma)+|\negg(\sigma)|,
\end{eqnarray} for all $\sigma \in S^B_n$, is  equidistributed with $\ell_B$ (see \cite[Corollary 4.6]{AdinBrenti}). Moreover $\fmaj(e)=0$ and $\fmaj(w_0)=n^2=\ell_B(w_0)$, so $\fmaj\in [\ell_B]$. For $2\leqslant n \leqslant 8$ a computer calculation yields
\begin{equation}\label{successione-fmaj}
  |\imag(\ell_B-\fmaj)|=3, 7, 15, 25, 39, 55, 75
\end{equation} and then $k^-(\ell_B,\fmaj)=5, 11, 17, 25, 33, 43, 53$.
Another computation shows that in $S_3^B$
$$\sum \limits_{\substack{\sigma\in S_3^B, \\\ell_B(\sigma)\neq 0}} \frac{\fmaj(\sigma)}{\ell_B(\sigma)} =\frac{22303}{420}\neq \frac{14731}{280}=\sum \limits_{\substack{\sigma\in S_3^B, \\  \fmaj(\sigma)\neq 0}} \frac{\ell_B(\sigma)}{\fmaj(\sigma)};$$ therefore, by Proposition \ref{criterioinv} we can conclude that in general there does not exist an involutions
$\iota : S_n^B \rightarrow S_n^B$ such that $\fmaj = \ell_B \circ \iota$. This also implies that the pair $(\fmaj,\ell_B)$ is not symmetric in general.

We end this section by proving that the functions $\ell_B$ and $\inmaj$ are equidistributed on any set $D_{I,K}$, where we have defined
$$D_{I,K}:=\{\sigma\in S^B_n: D_R(\sigma)\setminus \{s_0\}= I, \negg(\sigma^{-1})=K\},$$ for any $I\subseteq  S\setminus \{s_0\}$ and $K \subseteq [n]$.

\begin{thm} \label{ellinmaj}
 Let $n\in \mathbb{P}$. Then $$\sum \limits_{\sigma \in D_{I,K}} x^{\ell_B(\sigma)}=\sum \limits_{\sigma \in D_{I,K}} x^{\nmaj^*(\sigma)},$$ for all
$I\subseteq  S\setminus \{s_0\}$ and $K \subseteq [n]$.
\end{thm}
\begin{proof} Observe that the function $\ell_B\circ P^J$ assumes the constant value $\tilde{k}:=\sum \limits_{k\in K}k$ on $D_{I,K}$.
Therefore by Proposition \ref{propind}, Proposition \ref{prop-nmaj} and Equation \ref{ellimaj} we obtain

\begin{eqnarray*}
  \sum \limits_{\sigma \in D_{I,K}} x^{\ell_B(\sigma)} &=& \sum \limits_{\sigma \in D_{I,K}} x^{\ell_B(\sigma^J)+\ell_B(\sigma_J)}= x^{\tilde{k}}\sum \limits_{\sigma \in D_{I,K}} x^{\ell_B(\sigma_J)}\\
   &=& x^{\tilde{k}}\sum \limits_{\sigma \in D_{I,K}} x^{\maj^*(\sigma_J)} = \sum \limits_{\sigma \in D_{I,K}} x^{\ell_B(\sigma^J)+\maj^*(\sigma_J)} \\&=& \sum \limits_{\sigma \in D_{I,K}} x^{\nmaj^*(\sigma)}.
\end{eqnarray*}
\end{proof}
From the previous theorem we deduce the following equality in the polynomial semiring $\mathbb{N}[x,y]$.
\begin{cor} \label{cornmajfmaj}
   Let $n \in \mathbb{P}$. Then $$\sum \limits_{\sigma \in S^B_n} x^{\ell_B(\sigma)}y^{\fmaj(\sigma)}=\sum \limits_{\sigma \in S^B_n} x^{\nmaj^*(\sigma)}y^{\fmaj(\sigma)}.$$
\end{cor}
\begin{proof}
 The result follows by Theorem \ref{ellinmaj}, as for Equation \eqref{eqmajimaj}, since, by definition, the function $\fmaj$ is constant on every $D_{I,K}$ and these sets constitute a partition of $S^B_n$.
\end{proof}

In particular we find that $\ell_B+ \fmaj \sim \inmaj+ \fmaj$ and $\ell_B- \fmaj \sim \inmaj- \fmaj$. As in type $A$, by the criterium of Proposition \ref{criterioinv} one can see that there is no involution $\iota: S^B_n \rightarrow S^B_n$
such that $\ell_B\pm \fmaj= (\inmaj \pm \fmaj)\circ \iota$.

\section{Type D}
\label{sec5}

A Coxeter system of type $D_n$ can be realized with the group $S^D_n$ of bijections $w$ of the set $[\pm n]$  such that $w(-a)=-w(a)$ (identified with sequences of $n$ numbers with values in $[\pm n]$) and $\negg(w)\equiv 0 \mod 2$, generated by the set of involutions $S=\{s_0,s_1,...,s_{n-1}\}$,
 where $$s_i=(1,...,i-1,i+1,i,i+2,...,n)$$ and $$s_0=(-2,-1,3,...,n),$$ which satisfy the relations
$$s_is_js_i=\left\{
              \begin{array}{ll}
                s_js_is_j, & \hbox{if $|i-j|=1$ and $i,j\neq 0$ or $i=2$ and $j=0$;} \\
                s_j, & \hbox{if $|i-j|>1$ and $i,j\neq 0$ or $i\neq 2$ and $j=0$.}
              \end{array}
            \right.
$$ We denote by $\ell_D$ the length function relative to this presentation. The generator $s_0$ acts in the right of a sequence $(\sigma(1),...,\sigma(n))$ by interchanging $\sigma(1)$ with $-\sigma(2)$ and $\sigma(2)$ with $-\sigma(1)$, and in the left by
 interchanging  $1$ with $-2$ and $2$ with $-1$. We have that $$w_0=\left\{
                                                                      \begin{array}{ll}
                                                                        (-1,-2,-3,...,-n), & \hbox{if $n$ is even;} \\
                                                                        (1,-2,-3,...,-n), & \hbox{if $n$ is odd.}
                                                                      \end{array}
                                                                    \right.
$$ See \cite[Chapter 8]{bjornerbrenti} for further details. So identifying the elements of $S^D_n$ with sequences of $n$ integers subjects to these conditions, if $J=\{s_1,...,s_{n-1}\}$,
by Lemmas \ref{lemDR} and  \ref{TvJ} we see that $D(\sigma)=D_R(\sigma_J)$ and $\inv(\sigma)=|T(\sigma_J)|=\ell_D(\sigma_J)=\ell(\sigma_J)$,
for all sequences $\sigma \in S^D_n$. Therefore, since

\begin{eqnarray} \label{eqlD}
 \nonumber 
  \ell_D(\sigma) = \inv(\sigma)-\sum \limits_{i\in \negg(\sigma)}\sigma(i)-|\negg(\sigma)|,
\end{eqnarray} we find that
\begin{equation}\label{ell-PJ}
  \ell_D(\sigma^J)=-\sum \limits_{i\in \negg(\sigma)}\sigma(i)-|\negg(\sigma)|
\end{equation}
and $\ell_D(w_0)=n(n-1)$.

In \cite{Biagioli} was introduced a function $\dmaj \in F(S^B_n)$ equidistributed with $\ell_D$, the \emph{D-negative major index}.
This is defined by
\begin{eqnarray} \label{eqnmaj}
 \nonumber 
  \dmaj(\sigma) = \maj(\sigma)-\sum \limits_{i\in \negg(\sigma)}\sigma(i)-|\negg(\sigma)|,
\end{eqnarray} for all $\sigma \in S^D_n$. By Equation \eqref{ell-PJ} we have that $\dmaj=\maj+\ell_D \circ P^J$. The following proposition states that $\dmaj$ is induced by $\maj \in F(S_n)$, in the sense of Definition \ref{funzione indotta}. 

\begin{prop} \label{prop-dmaj}
  The function $\dmaj\in F(S^D_n)$ is induced by $\maj \in F(S_n)$, for all $n\geqslant 4$.
\end{prop}

By the existence of the involution $\iota : S_n \rightarrow S_n$ such that $\maj =\ell \circ \iota$,  we obtain, as in Corollary \ref{involB}, an involution on $S_n^D$ relating $\dmaj$ and $\ell_D$.

\begin{cor}
  There exists an involution $\tilde{\iota} : S_n^D \rightarrow S_n^D$ such that $\dmaj=\ell_D \circ \tilde{\iota}$.
  This involution satisfies $\tilde{\iota}(\sigma)=P^J(\sigma)\iota(\sigma_J)$, for all $\sigma \in S_n^D$.
\end{cor}

As a direct consequence of Theorem \ref{simmetria} and Proposition \ref{prop-dmaj} we obtain the result of \cite[Proposition 5.3]{Biagioli2}.
\begin{cor}
  The pair $(\dmaj,\ell_D)$ is symmetric.
\end{cor}

Knowing that $\dmaj\in F(S^D_n)$ is induced by $\maj \in F(S_n)$ and the behaviour of $\imag(\ell+\maj)$ and $\imag(\ell-\maj)$ on $S_n$,
we can deduce the behaviour of $\imag(\ell_D+\nmaj)$ and $\imag(\ell_D-\nmaj)$.

\begin{prop} \label{propD}
   Let $n\geqslant 4$ and $\ell_D, \dmaj : S^D_n \rightarrow \mathbb{N}$ be the functions length and D-negative major index. Then
  $$|\imag(\ell_D+\dmaj)|= \left\{
                              \begin{array}{ll}
                                3, & \hbox{if $n=4$;} \\
                                2n(n-1)-1, & \hbox{otherwise,}
                              \end{array}
                            \right.
  $$ and $\imag(\ell_D-\dmaj)=\imag(\ell-\maj)$, being $\maj \in F(S_n)$ the major index.
\end{prop}
\begin{proof}
  The case $n=4$ is obtained by a direct computation. Let $n>4$. Since $\dmaj$ is induced by $\maj$ and $\ell_D(w_0)=n(n-1)$, the result follows by Theorems \ref{teorema+}, \ref{teorema-} and \ref{teoremaA}.
\end{proof} In \cite{BiagiolCaselli} is defined another function $\Dmaj \in F(S^D_n)$, which is  equidistributed with $\ell_D$.
It satisfies $$\Dmaj(\sigma(1),...,\sigma(n))=\fmaj(\sigma(1),...,|\sigma(n)|),$$ for all $(\sigma(1),...,\sigma(n))\in S_n^D$. Note that $\Dmaj \not \in [\ell_D]$ since
$\Dmaj(w_0)\neq \ell_D(w_0)$. A direct computation shows that in $S_4^D$ we have
$$\sum \limits_{\substack{\sigma\in S_4^D, \\\ell_D(\sigma)\neq 0}} \frac{\Dmaj(\sigma)}{\ell_D(\sigma)} =\frac{6451033}{27720}\neq \frac{829573}{3465}=\sum \limits_{\substack{\sigma\in S_4^D, \\  \Dmaj(\sigma)\neq 0}} \frac{\ell_D(\sigma)}{\Dmaj(\sigma)};$$ therefore, by Proposition \ref{criterioinv} we can conclude that in general there does not exist an involutions
$\iota : S_n^D \rightarrow S_n^D$ such that $\Dmaj = \ell_D \circ \iota$.
This also implies that the pair $(\Dmaj,\ell_D)$ is not symmetric in general.

\section{Other types} \label{other type}

In this section we briefly discuss what kind of functions equidistributed with the length can be defined for other finite Coxeter systems, by induction from a parabolic subgroup.
We follows  \cite[Appendix A1]{bjornerbrenti} for the classification of the irreducible finite Coxeter systems.

In general there are many choices to
inject a Coxeter group in another Coxeter group as a parabolic subgroup. Also the Coxeter presentation of a Coxeter group is not unique in general (see \cite{BahlsTheIsomorphism} for results on the rigidity of a Coxeter group). Such choices of an injection and a presentation permit to realize different induced functions. Note that since the length depend on the presentation, different presentations give functions equidistributed with different ranks. One can easily see that for the reducible Coxeter systems the sum of functions equidistributed with the length on the irreducible components gives a function equidistributed with the length on the entire group. Therefore in these cases the induction is quite trivial and we can restrict our attention to the irreducible cases (see Remark \ref{oss prodotto cartesiano}). Nevertheless, the induction to an irreducible Coxeter system $(W,S)$ from a reducible one $(W_J,J)$ is not trivial.

In the dihedral groups $I_2(m)$ the only choice of a non trivial parabolic subgroup is $S_2$ and  the major index on $S_2$ coincides with the length; so the induced function is the length itself.

In the sequel we list, for any irreducible finite Coxeter system, the parabolic subgroups whose functions equidistributed with the length induce functions on the entire group equidistributed with the length. We omit the cases when the induced function is the length itself and we consider only irreducible subsystems $(W_J,J)$.

\begin{itemize}
  \item $E_6$: $S_3$, $S_4$, $S_5$, $S_6$, $S_4^D$ and $S_5^D$.
  \item $E_7$: $S_3$, $S_4$, $S_5$, $S_6$, $S_7$, $S_4^D$, $S_5^D$ and $E_6$.
  \item $E_8$: $S_3$, $S_4$, $S_5$, $S_6$, $S_7$, $S_8$, $S_4^D$, $S_5^D$, $E_6$ and $E_7$.
  \item $F_4$: $S_3$, $S^B_2$ and $S_3^B$.
  \item $H_3$: $S_3$ and $I_2(5)$.
  \item $F_4$: $S_3$, $S^B_2$ and $S_3^B$.
\end{itemize}

Then, for example, $\maj \in F(S_5)$ induces a function $\maj_{S_5,E_6} \in F(E_6)$ equidistributed with $\ell_{E_6}$,
$\dmaj\in F(S_5^D)$ induces a function $\dmaj_{S_5^D,E_8} \in F(E_8)$ equidistributed with $\ell_{E_8}$ and
$\fmaj \in F(S_3^B)$ induces a function $\fmaj_{S_3^B,F_4}$ on the group $F_4$ equidistributed with $\ell_{F_4}$.

\section*{Acknowledgements}
I would like to thank the anonymous referee for the meticulous revision and for some useful remarks.

\end{document}